%% file: KMcCormick_Matrix_Bundles.tex
\documentclass{amsart}[12pt,a4paper]
\usepackage{fullpage}
\usepackage{color}
\usepackage{amsmath}
\usepackage{amssymb}
\usepackage{amsfonts}
\usepackage{amsthm}
\usepackage{thmtools} 
\usepackage{thm-restate} 
\usepackage{hyperref} 
\usepackage[capitalize]{cleveref} 
\crefname{defn}{Definition}{Definitions} 
\crefname{restatable}{Theorem}{Theorems} 
\usepackage{enumerate}

\theoremstyle{definition}
\newtheorem{dummy}{Theorem}[section] 
\newtheorem{theorem}[dummy]{Theorem}
\newtheorem{defn}[dummy]{Definition}
\newtheorem{corollary}[dummy]{Corollary}
\newtheorem{lemma}[dummy]{Lemma}

\newtheorem{remark}[dummy]{Remark}

\newcommand{\inv}{^{-1}}
\newcommand{\Hom}{\text{Hom}}
\newcommand{\End}{\text{End}}
\newcommand{\eval}{\text{eval}}

\title{Matrix bundles and operator algebras over a finitely bordered Riemann surface}
\author{Kathryn McCormick}

\begin{document}

\maketitle
\begin{abstract} This note presents an analysis of a class of operator algebras constructed as cross-sectional algebras of flat holomorphic matrix bundles over a finitely bordered Riemann surface. These algebras are partly inspired by the bundle shifts of Abrahamse and Douglas. The first objective is to understand the boundary representations of the containing $C^*$-algebra, i.e.~Arveson's noncommutative Choquet boundary for each of our operator algebras. The boundary representations of our operator algebras for their containing $C^*$-algebras are calculated, and it is shown that they correspond to evaluations on the boundary of the Riemann surface. Secondly, we show that our algebras are Azumaya algebras, the algebraic analogues of $n$-homogeneous $C^*$-algebras.
	\end{abstract}

\begin{section}{Introduction}

The motivation for this work comes from problems in the structure theory of nonselfadjoint operator algebras. We are first and foremost concerned with boundary representations, which were introduced in 1969 by Arveson \cite{Arv1969a} as a noncommutative Choquet boundary for operator algebras. Since then, progress has been made toward calculating such representations; for a (non-exhaustive) sample of such results, see \cite{Hopenwasser1973}, \cite{Muhly1998a}, and \cite{DK2010}. Secondly, we were drawn to the theory bundle shifts. Initiated by Abrahamse and Douglas in \cite{Abrahamse1976}, bundle shifts are a class of subnormal operators for which properties such as invariant subspaces and generating $C^*$- and $W^*$- algebras can be determined by the shift operator's associated vector bundle. Such a shift operator acts on a Hilbert space of cross-sections via multiplication by holomorphic functions on the base space. In this note we present a class of operator algebras that draws on technology similar to that of bundle shifts, but features matrix bundles. Here, the cross-sections are matrix-valued and themselves are the operators. Thus while our matrix bundles have the holomorphic properties of the vector bundles associated to the bundle shift, they are more structurally tied to $n$-homogeneous $C^*$-algebras as studied, for example, in  \cite{Tomiyama1961}.

Throughout, $\overline{R}$ will denote a bordered Riemann surface as in \cite[II.3A]{AhlforsSario60}, and $R$ will denote its interior. We will assume that $\overline{R}$ is compact, and that the boundary $\partial R$ is made up of finitely many disjoint, analytic curves. Further, we will require that $R$ is hyperbolic, i.e.~the universal covering space of $R$ is the open unit disk $D$. Let $\pi_1(R)$ be the fundamental group of $R$, viewed as the group of covering transformations of $D$. As we will describe in \cref{Preliminaries}, we may view $R$ embedded in its double, so that the free and proper action of $\pi_1(R)$ on $D$ extends to a free and proper action of $\pi_1(R)$ on the universal cover of $\overline{R}$, $\widetilde{D}$. Thus the quotient map, $\pi : \widetilde{D} \to \overline{R}$, maps $D$ onto $R$ and $\widetilde{D} \setminus D$ onto $\partial R$.
 
View the projective unitary group, $PU_n(\mathbb{C})$, as the $*$-automorphism group of the $n \times n$ matrices $M_n(\mathbb{C})$, and take a representation $\rho : \pi_1(R) \to PU_n(\mathbb{C})$. Then we may construct a flat matrix $PU_n(\mathbb{C})$-bundle, $\mathfrak{E}(\overline{R})=\mathfrak{E}_\rho(\overline{R}):=(\widetilde{D} \times M_n (\mathbb{C}) )/\pi_1(R)$, where $\pi_1(R)$ acts via the formula $(x,A)\cdot g=(x \cdot g, \rho(g\inv)A)$. As we will discuss in \cref{Preliminaries}, $\mathfrak{E}(\overline{R})$ is a topological coordinate matrix $PU_n(\mathbb{C})$-bundle over $\overline{R}$ whose restriction to $R$ is a holomorphic matrix bundle. For $T \subseteq \overline{R}$, we write $\mathfrak{E}(T)$ for the restriction to the subset $T$. 

Our two primary objects of study are the $C^*$-algebra of continuous cross sections of $\mathfrak{E}(\overline{R})$, $\Gamma_c(\overline{R},\mathfrak{E}(\overline{R}))$, and the subalgebra of the continuous cross sections that are holomorphic on the interior $R$, $\Gamma_h(\overline{R},\mathfrak{E}(\overline{R}))$. The two main results of this paper develop the structure theory of the algebra $\Gamma_h(\overline{R},\mathfrak{E}(\overline{R}))$. For the first, note that $\Gamma_c(\overline{R}, \mathfrak{E}(\overline{R}))$ is an $n$-homogeneous $C^*$-algebra with spectrum $\overline{R}$, and thus all its irreducible  representations are given by composing evaluation at points $t \in \overline{R}$ with a $*$-isomorphism of the fibre over $t$ with $M_n(\mathbb{C})$. Define $\eval_t$ to be the evaluation at $t$ composed with one such choice of isomorphism.

\begin{restatable}{theorem}{BoundaryRepsAreEval} \label{BoundaryRepsAreEval} The boundary representations (in the sense of Arveson \cite{Arv1969a}) of $\Gamma_c(\overline{R},\mathfrak{E}(\overline{R}))$ for $\Gamma_h(\overline{R},\mathfrak{E}(\overline{R}))$ are precisely the evaluation maps $\eval_t$, where $t \in \partial R$.
	\end{restatable}

As a result, we have:

\begin{corollary} The $C^*$-envelope of $\Gamma_h(\overline{R},\mathfrak{E}(\overline{R}))$ is $\Gamma_c(\partial R,\mathfrak{E}(\partial R))$.
	\end{corollary}

A special case of \cref{BoundaryRepsAreEval} is proven in \cite{Hopenwasser1973} by Hopenwasser: Let $A(\overline{D})$ be the disk algebra, i.e.~the subalgebra of continuous functions on the closed unit disk $\overline{D}$ that are holomorphic on the interior, and let $\mathfrak{E}(\overline{D})$ be the trivial matrix bundle over $\overline{D}$. Then $\Gamma_c(\overline{D},\mathfrak{E}(\overline{D}))\simeq C(\overline{D}) \otimes M_n (\mathbb{C})$ is a $C^*$-algebra, $\Gamma_h(\overline{D},\mathfrak{E}(\overline{D})) \simeq A(\overline{D})\otimes M_n (\mathbb{C})$ is a subalgebra, and the boundary representations of $\Gamma_c(\overline{D},\mathfrak{E}(\overline{D}))$ for $\Gamma_h(\overline{D},\mathfrak{E}(\overline{D}))$ are given by evaluation at points $t$ on the boundary of $D$, the circle.

The second theorem concerns the structure of $\Gamma_h(\overline{R},\mathfrak{E}(\overline{R}))$ as an Azumaya algebra. An Azumaya algebra, also known as a central separable algebra, can be regarded as the purely algebraic form of an $n$-homogeneous $C^*$-algebras \cite[pg.~532]{Artin1969}. The second theorem, therefore, is not too surprising, but also not easily demonstrated without the help of \cite[Thm.~2.1.5]{Procesi1972}. We will discuss Azumaya algebras more thoroughly in \cref{AzumayaAlgebrasSection}.

\begin{restatable}{theorem}{HoloSectionsAlgIsAzumaya} \label{HoloSectionsAlgIsAzumaya} The algebra $\Gamma_h(\overline{R},\mathfrak{E}(\overline{R}))$ is an Azumaya algebra of rank $n$.
	\end{restatable}

Thus the structure of the nonselfadjoint subalgebra $\Gamma_h(\overline{R},\mathfrak{E}(\overline{R}))$ reflects some of the structure of its ($n$-homogeneous) $C^*$-envelope.

	\end{section}

\begin{section}{Preliminaries} \label{Preliminaries}
Recall we consider the bordered Riemann surface $\overline{R}$ embedded in its double, $L$ \cite[II.3E]{AhlforsSario60}. By construction of the double, we see $\overline{R}$ is contained in a domain $\breve{R}$, with $\overline{R} \subset \breve{R} \subset L$, where $\breve{R}$ contains annular neighborhoods of the boundary curves \cite[II.3]{AhlforsSario60}. Further, we may choose $\breve{R}$ in such a way that $\pi_1(R) \simeq \pi_1(\breve{R})$, and these two fundamental groups will be isomorphic to a free group on finitely many generators \cite[Ch.~6, Sec.~5, pg.~199]{Massey1967}.

By the hypothesis, the universal covering space of $R$ is taken to be the open unit disk $D$. Further, $D$ can be realized as a subdomain of the universal covering space of $\breve{R}$ that we call $\breve{D}$. The fundamental group of $R$, $\pi_1(R)$, is a group of covering transformations of $D$, and thus acts freely, properly, and holomorphically on $D$, as does $\pi_1(\breve{R})$ on $\breve{D}$. This gives $\pi : D \to R$ and $\breve{\pi} : \breve{D} \to \breve{R}$ the structure of $\pi_1(R)$- and $\pi_1(\breve{R})$- principal bundles, respectively \cite[Prop.~I.14.6-7]{Steenrod1951}. We can in fact say more; the isomorphism $\pi_1(R) \simeq \pi_1(\breve{R})$ is defined in such a way that the action of $\pi_1(\breve{R})$ leaves $D$ invariant, and the action coincides with the action of $\pi_1(R)$ when restricted to $D$. If $r \in \partial R$, then any point in the universal cover over $r$ must be on the boundary of $D$ in $\breve{D}$. Let $\widetilde{D}$ be the subset of $\breve{D}$ that covers $\overline{R}$, so that $\widetilde{D}\setminus D$ is the subset that covers $\partial R$. (Note that this might not be $\mathbb{T}$. In most cases the complement of $\widetilde{D}\setminus D$ in $\mathbb{T}$ is a Cantor set. See e.g.,~\cite[pg.~67]{Katok1992}). Then $\pi_1(\breve{R})$ restricts to act on $\widetilde{D}\setminus D$ continuously, freely, and properly, which gives $\pi : \widetilde{D}\setminus D \to \partial R$ the structure of a $\pi_1(\breve{R})$-principal bundle as well. 

As discussed in the introduction, we will build a fibre bundle over $\overline{R}$ from a representation of $\pi_1(R)$, but we will this bundle primarily as a coordinate fibre bundles as described by Steenrod \cite{Steenrod1951}. For particular applications of fibre bundles to Riemann surfaces, we follow Gunning \cite{Gunning1967}.

Consider a coordinate bundle $\mathfrak{B}=(B,\pi, S, M,G)$ as described in \cite[2.3]{Steenrod1951}, where $B$ is the bundle space, $S$ is the base space, $\pi$ is the projection map from $B$ onto $S$, $M$ is the fibre, and $G$ is the structure group that acts effectively on $M$. Further, the bundle $\mathfrak{B}$ comes with a family of open sets $\mathcal{U}$ contained in $S$ and fibre-preserving homeomorphisms $\phi_U : U \times M \to B|_U$, $U \in \mathcal{U}$, such that for any fixed $x \in U \cap V$, $\phi_U^{-1}(x,\cdot) \cdot \phi_V(x,\cdot)=:g_{UV}(x)$ coincides with the action of an element in $G$. The maps $g_{UV} : U \cap V \to G$ are called \emph{transition functions}. Two bundles $\mathfrak{B_i}=(B_i,\pi_i, S_i, M,G)$, $i=1,2$, are called $G$-equivalent if there is a homeomorphism $F: B_1 \to B_2$ such that $\pi_2 \circ F = \pi_1$ and $F(b\cdot g)=F(b) \cdot g$ for every $b \in B_1$, $g \in G$. 

Equivalently, and this is the perspective we will stress, we may start from the open cover $\mathcal{U}$ and define a \emph{principal coordinate $G$-bundle} as a collection of transition functions $(g_{UV})_{U,V \in \mathcal{U}}$ where $g_{UV}: U \cap V \to G$ are continuous maps that satisfy $g_{UV}g_{VW}=g_{UW}$ \cite[Sec.~3]{Steenrod1951}. The transition functions are 1-cocycles for the cover $\mathcal{U}$ with values in $G$, and $(g_{UV})_\mathcal{U}$ and $(g'_{U'V'})_\mathcal{U'}$ are called \emph{$G$-equivalent} if there is a common refinement of covers $\mathcal{W}$ and a 0-cochain $(\lambda_W)_\mathcal{U}$ for the cover with values in $G$ so that $g'_{XW}=\lambda_{X}\inv g_{XW} \lambda_W$ for every $X, W \in \mathcal{W}$, i.e.~$(g_{XW})_\mathcal{W}$ is cohomologous to $(g'_{XW})_\mathcal{W}$ via $(\lambda_W)_\mathcal{W}$. If $H$ is a subgroup of $G$ and the transition functions of $(g_{UV})_\mathcal{U}$ map into $H$ we may think of the principal bundle as a $G$-bundle or an $H$-bundle; two $H$-equivalent bundles then must be implemented by a cochain with values in $H$. A \emph{principal $G$-bundle} is a $G$-equivalence class of principal coordinate $G$-bundles over $S$. Another perspective of a coordinate fibre bundle is given in terms of sheaves, which will be useful to us as well. If we let $\mathcal{G}_c$ be the sheaf of germs of continuous $G$-valued functions on $S$, then a principal $G$-bundle is an element of the sheaf cohomology set $H^1(S;\mathcal{G}_c)$ and a principal coordinate $G$-bundle associate to a cover $\mathcal{U}$ will be identified as an element in $Z^1(\mathcal{U};\mathcal{G}_c)$. See Raeburn and Williams \cite[4.53-4.54]{Raeburn1998}, who treat this sheaf-theoretic perspective, or see  \cite{Grothendieck1955} for an earlier but still very relevant account.

If $S$ is a Riemann surface, if $G$ is a topological group with a complex analytic structure, and if the cocycles are holomorphic maps, then we may similarly define the \emph{holomorphic principal coordinate $G$-bundle} determined by the data $(g_{UV})_\mathcal{U}$. If we let $\mathcal{G}_h$ be the sheaf of germs of holomorphic $G$-valued functions on $S$, then a holomorphic principal $G$-bundle is an element of the set $H^1(S; \mathcal{G}_h)$. 

Given an element $(g_{UV})_\mathcal{U} \in Z^1(\mathcal{U};\mathcal{G}_c)$ along with the continuous and effective action of the structure group $G$ on a topological space (or complex manifold) $M$, then we define the \emph{coordinate $M$-fibre $G$-bundle} determined by the data $((g_{UV})_\mathcal{U},M)$. Two coordinate $M$-fibre bundles are called \emph{$G$-equivalent} if their underlying principal bundles are $G$-equivalent and they have conjugate actions of $G$ on $M$. A fibre bundle is an equivalence class of such coordinate bundles. If $S$ is a Riemann surface, if $G$ a topological group with complex structure, and if $M$ a complex manifold, we may define analogously define a \emph{holomorphic coordinate fibre bundle} and a \emph{holomorphic fibre bundle}.

The \emph{product principal bundle} over a space $S$ is the equivalence class of the coordinate principal bundle with a single transition function $g_{SS} : S \to 1_G$, where $1_G$ is the unit of $G$. We say a fibre bundle is \emph{$G$-trivial} if its underlying principal bundle has a coordinate representative that is $G$-equivalent to the product bundle.

Suppose $G$ is a subgroup of $PGL_n(\mathbb{C})$, where $PGL_n(\mathbb{C})$ is viewed as the group of algebra automorphisms of $M_n(\mathbb{C})$. We define a \emph{matrix $G$-bundle} to be a fibre bundle with structure group $G$ and fibres $M_n (\mathbb{C})$ on which the group $G$ acts by algebra automorphisms. For most of this paper $G$ is the subgroup $PU_n(\mathbb{C})$. The following discussion could be slightly shortened if $PU_n(\mathbb{C})$ were the only subgroup. However there is a crucial moment in which we deal with $PGL_n(\mathbb{C})$-bundles, namely, when we invoke Grauert's theorem, and thus we present a unified approach to the theory.

We now wish to construct the flat coordinate matrix $PU_n(\mathbb{C})$-bundles described in the introduction (and defined in \cref{DefOfOurBundles} below), which are determined by a representation of the fundamental group of a surface. We have not been able to find such a construction for matrix bundles explicitly written down in the literature. However, this may be accomplished through minor modifications of the theory for vector bundles, which we have cited in the following lemmas. For the interested reader, we include a discussion of the proofs as a short appendix.

For a group $G$ and space $S$, denote by \underline{$G$} the constant sheaf of $G$-valued functions on $S$ (i.e.~germs of \emph{locally constant}, $G$-valued functions), denote by $\mathcal{G}_c$ the sheaf of germs of \emph{continuous} $G$-valued functions on $S$, and denote by $\mathcal{G}_h$ the sheaf of germs of \emph{holomorphic} $G$-valued functions on $S$. In the following definitions we let $i$ be either the inclusion map $i : \mbox{\underline{$G$} } \to \mathcal{G}_c$ or $i: \mbox{\underline{$G$}} \to \mathcal{G}_h$, respectively.
\begin{defn} Let $G$ be a subgroup of $PGL_n(\mathbb{C})$ and $S$ a topological space. Then a \emph{flat continuous matrix $G$-bundle} is a matrix $G$-bundle for which there is a coordinate representative with locally constant transition functions. \label{DefFlatContMatrixBundle}
	\end{defn}
It is customary to identify a flat matrix $G$-bundle with an element in the image of the map $i_* : H^1(S;\mbox{\underline{$G$}}) \to H^1(S;\mathcal{G}_c)$.
\begin{defn} Let $G$ be a subgroup of $PGL_n(\mathbb{C})$, and let $S$ be a Riemann surface. Then a \emph{flat holomorphic matrix $G$-bundle} is a holomorphic matrix $G$-bundle for which there is a coordinate representative with locally constant transition functions. Such a coordinate representative is also referred to as flat. \label{DefFlatHoloMatrixBundle}
	\end{defn}
It is customary to identify a flat holomorphic matrix $G$-bundle with an element in the image of the map $i_* : H^1(S;\mbox{\underline{$G$}}) \to H^1(S;\mathcal{G}_h)$. 

Note that if we require that the matrix bundle be both holomorphic and a $PU_n(\mathbb{C})$-bundle, then the bundle will be flat (i.e.~satisfy \cref{DefFlatHoloMatrixBundle}) since the holomorphic transition functions must have values in $PU_n(\mathbb{C})$. While a product bundle (i.e.~a trival bundle) is flat, not all flat bundles are trivial. This is seen clearly in the following lemma, which is a well-known fact in a variety of contexts. For a version in the case of flat holomorphic vector bundles, see for example \cite[Lem.~27]{Gunning1966}; in the case of flat continuous fibre bundles, see for example \cite[Thm.~13.9]{Steenrod1951}; and \cite[2.6]{Kobayashi1987} has a standard version for flat differentiable vector bundles.

Let $\Hom(\pi_1(S),G)$ denote all of the homomorphisms of $\pi_1(S)$ into $G$, and let $\Hom(\pi_1(S),G)/G$ denote the homomorphisms modulo the action of $G$ by conjugation via $(\varphi \cdot g )(\gamma) = \varphi(\gamma) \cdot g$ for every $\gamma \in \pi_1(S)$.

\begin{lemma} Any flat coordinate matrix $G$-bundle over $S$ as in \cref{DefFlatContMatrixBundle} is determined by an element in $\Hom(\pi_1(S),G)$, and every element in $\Hom(\pi_1(S),G)$ determines a flat matrix $G$-bundle. If two elements $\rho_1,\rho_2 \in \Hom(\pi_1(S),G)$ are equivalent under the action of $G$ then they define equivalent flat coordinate matrix bundles, and for any two equivalent flat coordinate matrix bundles $\mathfrak{B}_i$, $i=1,2$ there exists $\rho_i \in \Hom(\pi_1(S),G)$ determining $\mathfrak{B}_i$, $i=1,2$, so that $\rho_1$ and $\rho_2$ are equivalent.\label{FlatPrincipalBundleAsRep}
	\end{lemma}

\begin{lemma} Two flat \emph{holomorphic} coordinate matrix $PU_n(\mathbb{C})$-bundles over a Riemann surface $S$ are holomorphically $PU_n(\mathbb{C})$-equivalent iff for any determining representation $\rho_1$ of the first bundle and $\rho_2$ of the second bundle, $\rho_1$ and $\rho_2$ are equivalent in $\Hom(\pi_1(S),PU_n(\mathbb{C})))$. \label{FlatMatrixBundleAsRep}
	\end{lemma}

In our work we need to extend bundles over the interior of the bordered Riemann surface $\overline{R}$ to bundles over the slightly larger surface $\breve{R}$, where $\breve{R}$ is as described in the beginning of this section. We will do so using the following lemma.

\begin{lemma} \cite[pg.~306]{Widom1971} Let $H$ be a subgroup of $PU_n(\mathbb{C})$. Any flat matrix $H$-bundle over $R$ with coordinate representative $(g_{UV})_\mathcal{U}$ may be extended to a flat matrix $H$-bundle  over $\breve{R}$ with coordinate representative $(h_{U'V'})_\mathcal{U'}$ so that $U=U' \cap R$ for every $U \in \mathcal{U}$ and $g_{UV}=h_{U'V'}|_R$. \label{ExtensionOfBundleToRegularRegion}
	\end{lemma}

Thus we make the following defintion, which describes the bundles we will ultimately be working with and the relationship between their restrictions to subsets.

\begin{defn} \label{DefOfOurBundles} By \cref{FlatMatrixBundleAsRep}, a representation $\rho: \pi_1(R) \to PU_n(\mathbb{C})$ defines a coordinate holomorphic matrix $PU_n(\mathbb{C})$-bundle over $R$. Extend this bundle to a bundle over $\breve{R}$, also defined by $\rho$, using \cref{ExtensionOfBundleToRegularRegion}. Name the resulting coorindate bundle $\mathfrak{E}_\rho(\breve{R})$. Let $\mathfrak{E}_\rho(R)$ be the restriction back to the flat holomorphic matrix $PU_n(\mathbb{C})$-bundle over $R$, and let the restriction to the flat continuous coordinate matrix $PU_n(\mathbb{C})$-bundle over $\overline{R}$ or $\partial R$ be denoted as $\mathfrak{E}_\rho(\overline{R})$ and $\mathfrak{E}_\rho(\partial R)$, respectively. 
	\end{defn}

A \emph{continuous (or holomorphic) section} of a coordinate fibre bundle $\mathfrak{B}=(B,\pi, S, M,G)$ is a continuous (holomorphic) map $\sigma : S \to B$ so that $\pi \circ \sigma = 1_S$. When we view $\mathcal{B}$ as a pair $((g_{UV})_\mathcal{U},M)$, we may also think of a section $\sigma$ as a collection of maps $(\sigma_U)_\mathcal{U}$ where each $\sigma_U: U \to M$ is a continuous (or holomorphic) function such that $\sigma_U g_{UV} = \sigma_V$ for every $U,V \in \mathcal{U}$ \cite[pg.~177-178]{Gunning1966}. The latter relation defines a 1-1 correspondence between sections of two $G$-equivalent coordinate fibre bundles, namely, if $(\sigma_U)_\mathcal{U}$ is a (holomorphic) section of $((g_{UV})_\mathcal{U},M)$ and $(g_{UV})_\mathcal{U}$ is $G$-equivalent to $(g'_{UV})_\mathcal{U}$ via the cochain $(\lambda_U)_\mathcal{U}$, then $(\sigma_U \lambda_U\inv)_\mathcal{U}$ is a (holomorphic) section of $((g'_{UV})_\mathcal{U},M)$.

In \cite{Grauert1958T}, Grauert proved several fundamental results on holomorphic bundles over complex spaces. The particular theorem which we will need is the following \cite[Thm.~7]{Grauert1958T}, which was proven by Rohrl slightly earlier in \cite[Satz 3]{Rohrl1957}.

\begin{theorem}[Grauert-Rohrl Theorem] Let $G$ be a connected, complex Lie group and $S$ an open, connected Riemann surface. Then any fibre bundle over $S$ with group $G$ is holomorphically $G$-trivial.  \label{GrauertsTheorem}
	\end{theorem}

We will now apply Grauert's theorem in the case when $G=PGL_n(\mathbb{C})$ to show there exist many continuous holomorphic sections in the bundles we study. This is necessary for \cref{HoloSectionsGenerateCont} and\cref{HoloSectionsAlgIsAzumaya}) later.

\begin{lemma} Let $\mathfrak{E}_\rho(\overline{R})=((g_{UV})_\mathcal{U},M_n(\mathbb{C}))$, where $(g_{UV})_\mathcal{U}$ are restrictions of the transition functions of $\mathfrak{E}_\rho(\breve{R})$. For any point $r \in \overline{R}$, $A \in M_n$, and a prescribed open set $U_0$ containing $r$, we can find a continuous section $\sigma$ of the bundle over $\overline{R}$ so that $(\sigma_U)_\mathcal{U}$ is holomorphic over $R$ and $\sigma_{U_0}(r)=A$. \label{ExistenceOfSections}
	\end{lemma}
\begin{proof} Construct $((g_{UV})_\mathcal{U},M_n(\mathbb{C}))$ from $\mathfrak{E}_\rho(R)$ using \cref{ExtensionOfBundleToRegularRegion}. By \cref{GrauertsTheorem}, the bundle is holomorphically $PGL_n(\mathbb{C})$-trivial. Let the cochain $(\lambda_U)_\mathcal{U}$ implement the $PGL_n(\mathbb{C})$-bundle equivalence with the product bundle, and pick a section $(\sigma_U)_\mathcal{U}$ of the product bundle so that $\sigma_{U_0}(r)=A \lambda_{U_0}(r)$. Then $(\sigma_U \lambda_U\inv)$ is a holomorphic section of the bundle $((g_{UV})_\mathcal{U},M_n(\mathbb{C})$ over $\breve{R}$, and $(\sigma_{U\cap\overline{R}}\lambda_{U \cap \overline{R}})$ is a continuous section of the bundle over $\overline{R}$ that is holomorphic on $R$.
	\end{proof}

For the remainder of this paper, we say a section $(\sigma_U)_\mathcal{U}$ of the bundle $\mathfrak{E}_\rho(\overline{R})$ is \emph{continuous holomorphic} if it is a continuous section such that each $\sigma_U$ is holomorphic when restricted to $R$. Note that if $\mathfrak{E}_{\rho_1}(R)$ is holomorphically $PU_n(\mathbb{C})$-equivalent to $\mathfrak{E}_{\rho_2}(R)$ then a continuous holomorphic section of $\mathfrak{E}_{\rho_1}(\overline{R})$ will be mapped to a continuous holomorphic section of $\mathfrak{E}_{\rho_2}(\overline{R})$, and vice versa.

Let $\mathcal{U}$ be an open cover of $\breve{R}$, and let $U$ be an open set in $\mathcal{U}$ with $r \in U$. Let $(g_{UV})_\mathcal{U}$ be the transition functions of $\mathfrak{E}(\overline{R})$ and define the map $\eval_{\mathcal{U},U,r} : \Gamma_c(\overline{R},\mathfrak{E}(\overline{R})) \to M_n(\mathbb{C})$ by $\eval_{\mathcal{U},U,r}((\sigma_U)_\mathcal{U})=\sigma_U(r)$. It's easy to check that $\eval_{\mathcal{U},U,r}$ is an irreducible representation of $\Gamma_c(\overline{R},\mathfrak{E}(\overline{R}))$. If $U_1,U_2 \in \mathcal{U}$, then there is a cocycle $g_{U_1U_2}$ for which $\sigma_{U_1} g_{U_1U_2} = \sigma_{U_2}$, thus for $r \in U_1 \cap U_2$, $\eval_{\mathcal{U},U_1,r}$ will be unitarily equivalent to $\eval_{\mathcal{U},U_2,r}$. So given a fixed coordinate bundle $((g_{UV})_\mathcal{U},M)$ we abuse language and call this unitary equivalence class of representations evaluation at $r$ or $\eval_r$. Note that once a coordinate representative $((g_{UV})_\mathcal{U},M)$ is fixed, a section over $T\subseteq \breve{R}$ is completely determined by its evaluations at the points in $T$.

\end{section}

\begin{section}{The operator algebra of continuous holomorphic sections} \label{BoundaryRepsSection}

With point evaluations defined, we may proceed to define a $*$-algebra structure on the collection of continuous sections $\Gamma_c(\overline{R},\mathfrak{E}(\overline{R}))$ by pointwise addition and multiplication. Since the transition functions $g_{UV}$ are pointwise $*$-automorphisms, pointwise involution on a section  $\sigma = (\sigma_U)_\mathcal{U}$ is well-defined. For each $r \in \overline{R}$, the operator norm of $(\sigma_U)_\mathcal{U}(r)$ in $M_n(\mathbb{C})$ is also well-defined by these hypotheses on the transition functions. Set $\|(\sigma_U)\|:= \sup_{r\in U \subseteq \overline{R}} \|\sigma_U(r)\|$. There is a maximum modulus theorem in our setting: if $\sigma \in \Gamma_h(\overline{R},\mathfrak{E}(\overline{R}))$, then $\|(\sigma_U)\|:= \sup_{r\in U \cap \partial R} \|\sigma_U|_{\partial R}(r)\|$. For each $U \in \mathcal{U}$, $\sigma_U$ is a matrix of holomorphic functions, and $\| \sigma_U \|^2 = \|\sigma_U^* \sigma_U \|$ will be a subharmonic function in $U$. By hypotheses on the transition functions, $\|\sigma_U(r)\|^2=\|\sigma_V(r)\|^2$ when $r \in U \cap V$, and so $\|(\sigma_U)_\mathcal{U}\|^2$ is a subharmonic function on all of $R$ and takes its maximum on $\partial R$. Moreover, $\Gamma_h(\overline{R},\mathfrak{E}(\overline{R}))$ will be a closed subset of $\Gamma_c(\overline{R},\mathfrak{E}(\overline{R}))$ under the norm. We also see the following statement holds from the previous observations in \cref{Preliminaries}.

\begin{remark} If $\mathfrak{E}_{\rho_1}(\breve{R})$ is holomorphically $PU_n(\mathbb{C})$-equivalent to $\mathfrak{E}_{\rho_2}(\breve{R})$ then $\Gamma_h(\overline{R},\mathfrak{E}_{\rho_1}(R))$ is isomorphic to $\Gamma_h(\overline{R},\mathfrak{E}_{\rho_1}(R))$ as Banach algebras. Further, the $PU_n(\mathbb{C})$-equivalence of $\mathfrak{E}_{\rho_1}(R)$ and $\mathfrak{E}_{\rho_2}(R)$ implements a $C^*$-isomorphism of the continuous section algebras and so this Banach isomorphism of the continuous holomorphic algebras is a complete isometry.
	\end{remark}

The algebra $\Gamma_c(\overline{R},\mathfrak{E}(\overline{R}))$ is an $n$-homogeneous $C^*$-algebra with Hausdorff spectrum \cite[Thm.~5]{Tomiyama1961}. That is, all of the irreducible representations of $\Gamma_c(\overline{R},\mathfrak{E}(\overline{R}))$ are $n$-dimensional, and correspond to (representatives of the class) $\eval_r$.
We can see that the bundle $\mathfrak{E}(\overline{R})$ is an example of a continuous field of $C^*$-algebras as described in \cite[10.3.1]{dix}. Using this, we show:

\begin{lemma} The closed subalgebra $\Gamma_h(\overline{R},\mathfrak{E}(\overline{R}))$ generates $\Gamma_c(\overline{R},\mathfrak{E}(\overline{R}))$ as a $C^*$-algebra. \label{HoloSectionsGenerateCont}
	\end{lemma}
\begin{proof} This follows from the observation that $\mathfrak{E}(\overline{R})$ is a continuous field of $C^*$-algebras and the $*$-semigroup generated by $\Gamma_h(\overline{R},\mathfrak{E}(\overline{R}))$ is a total subset in $\Gamma_c(\overline{R},\mathfrak{E}(\overline{R}))$ \cite[10.2.3]{dix}. This is immediate from \cref{ExistenceOfSections} where we show every point in every fibre is a value of a holomorphic section.
	\end{proof}

\begin{lemma} The restriction map $res : \Gamma_c(\overline{R},\mathfrak{E}(\overline{R})) \to \Gamma_c(\partial R,\mathfrak{E}(\partial R))$ induces a completely isometric isomorphism from $\Gamma_h(\overline{R},\mathfrak{E}(\overline{R}))$ to a norm-closed subalgebra of $\Gamma_c(\partial R,\mathfrak{E}(\partial R))$, vis.~the set $\{ \sigma|_{\partial R} \mid \sigma \in \Gamma_h(\overline{R},\mathfrak{E}(\overline{R})) \}$. \label{RestrictionCompletelyIsom}
	\end{lemma}
	\begin{proof} The map $res$ is a well-defined, algebra homomorphism. From our earlier remarks on the definition of the section norm, it's easy to see that the map $res|_{\Gamma_h(\overline{R},\mathfrak{E}(\overline{R}))}$ is an isomorphism of closed subspaces by the maximum modulus theorem. We need that the maps $res_n: M_n(\Gamma_h(\overline{R},\mathfrak{E}(\overline{R}))) \to M_n(\Gamma_c(\partial R,\mathfrak{E}(\partial R)))$, where $res_n({(\sigma_{ij}}_U)_\mathcal{U})=(res(({\sigma_{ij}}_U)_\mathcal{U}))$, are each isometric. But this is true since the matrix norm composed with a holomorphic section is a subharmonic function, and thus attains its supremum on the boundary.
		\end{proof}
Consequently, we write $\Gamma_h(\partial R,\mathfrak{E}(\partial R))$ for $\{ \sigma \in \Gamma_c(\overline{R},\mathfrak{E}) \mid \sigma(\partial R)=0 \}$.

\begin{defn} \cite[2.1.1]{Arv1969a} Let $\mathcal{C}$ be a unital $C^*$-algebra and $\mathcal{A}$ be a closed unital subalgebra of $\mathcal{C}$ that generates $\mathcal{C}$ as a $C^*$-algebra. Then an irreducible representation $\pi: \mathcal{C} \to B(H)$ is called a \emph{boundary representation of $\mathcal{C}$ for $\mathcal{A}$} if the only completely positive map from $\mathcal{C}$ to $B(H)$ extending $\pi|_\mathcal{A}$ is $\pi$ itself.
	\end{defn}

\begin{defn} \cite[2.2.3]{Arv1969a} Let $\mathcal{A}$ generate the $C^*$-algebra $\mathcal{C}$ as above. The \emph{Shilov boundary ideal} is the largest ideal $\mathcal{I}$ of $\mathcal{C}$ such that the quotient map from $\mathcal{C}$ onto $\mathcal{C} / \mathcal{I}$ is completely isometric when restricted to $\mathcal{A}$. We denote the Shilov boundary ideal as $\mathcal{I}(\mathcal{A})$; we will always be clear what the containing $C^*$-algebra is.
	\end{defn}
The Shilov boundary ideal always exists. We can compute $\mathcal{I}(\mathcal{A})$ with the formula $\mathcal{I}(\mathcal{A})= \bigcap \{ \ker \pi | \; \pi \text{ a boundary representation of }\mathcal{C} \text{ for }\mathcal{A} \}$. This presentation was proven by Arveson in \cite{Arv2008a} for the separable case, and by Davidson and Kennedy \cite{DavidsonKennedy2015} for the general case.

\begin{defn} \cite[2.2.4]{Arv1969a} \label{C-starEnvelope} Let $\mathcal{C}$ be a unital $C^*$-algebra and let $\mathcal{A}$ be a closed, unital subalgebra generating $\mathcal{C}$. Then the \emph{$C^*$-envelope} of $\mathcal{A}$ is the quotient $\mathcal{C}/\mathcal{I}(\mathcal{A})$.
	\end{defn}
By \cite[Thm.~2.2.5]{Arv1969a} $C^*_e(\mathcal{A})$, has the following universal property: for any other $C^*$-algebra $\mathcal{C}_1$ generated by $\mathcal{A}$, there exists a surjective $*$-homomorphism $\varphi : \mathcal{C}_1 \to C^*_e(\mathcal{A})$ so that $\varphi \circ i_{\mathcal{C}_1} = i_{C^*_e(\mathcal{A})}$. Here $i_{\mathcal{C}_1}$ and $i_{C^*_e(\mathcal{A})}$ are the inclusion maps. 

The $C^*$-envelope $C^*_e(\mathcal{A})$ is independent of the containing $C^*$-algebra $\mathcal{C}$.

One of the main tools from noncommutative boundary theory which we will need in order to compute the $C^*$-envelope is the following theorem of Arveson that relates the boundary representations of two completely isometric operator algebras.

\begin{theorem} \cite[2.1.2]{Arv1969a} \label{BoundaryRepCorrespondence} For $i=1,2$, let $B_i$ be a unital $C^*$-algebra, and let $A_i$ be a closed linear subspace of $B_i$ containing the identity. Suppose $A_i$ generates $B_i$ as a $C^*$-algebra for $i=1,2$, and suppose $\varphi : A_1 \to A_2$ is a completely isometric linear map, such that $\varphi(1)=1$. Then for every boundary representation $\omega_1$ of $B_1$ for $A_1$ there is a boundary  representation $\omega_2$ of $B_2$ such that $\omega_2 \circ \varphi = \omega_1$ on $A_1$.
	\end{theorem}

The final ingredient that we need is an analogue of peak points in the commutative setting. These are called \emph{peaking representations}. We pause briefly to recall the commutative setting applied to the centers of the algebras $\Gamma_c(\overline{R},\mathfrak{E}(\overline{R}))$ and $\Gamma_h(\overline{R},\mathfrak{E}(\overline{R}))$. Sections in the center of $\Gamma_c(\overline{R},\mathfrak{E}(\overline{R}))$ will be the sections with values that lie solely in the center of $M_n(\mathbb{C})$, i.e.~$\mathcal{Z}(\Gamma_c(\overline{R},\mathfrak{E}(\overline{R})))=C(\overline{R})I_n \simeq C(\overline{R})$, and $\mathcal{Z}(\Gamma_h(\overline{R},\mathfrak{E}(\overline{R})))\simeq A(\overline{R})$, where $A(\overline{R})$ is the collection of continuous holomorphic functions from $\overline{R}$ into $\mathbb{C}$.

A \emph{peak point} of a uniform algebra $\mathcal{A} \subseteq C(X)$ is a point $x \in X$ such that there exists an $f \in \mathcal{A}$ so that $f(x)=1$ and $|f(y)|<1$ for all $y \not = x$. 
\begin{lemma} The peak points for $A(\overline{R}) \subset C(\overline{R})$ are all the points $r \in\partial R$. \label{PeakPointsForCenterAreOnBoundary}
	\end{lemma}
	\begin{proof} $A(\overline{R})|_{\partial R}$ is a hypo-Dirichlet algebra in the sense of Wermer \cite[pg.~162]{Wermer1964} and so $\partial R$ is the Choquet boundary of $A(\overline{R})$ by \cite[Thm.~3.1, Cor.~2]{AhernSarason1967}. Thus by \cite[Thm.~6.5]{Bishop1959}, $\partial R$ is the set of peak points for $A(\overline{R})$.
		\end{proof}

\begin{defn} \cite[Def.~7.1]{Arv2011} Let $\mathcal{C}$ be a $C^*$-algebra and $\mathcal{A}$ be a closed unital operator algebra that generates $\mathcal{C}$ as a $C^*$-algebra. Then an irreducible representation $\pi : \mathcal{C} \to B(H)$ is called \emph{peaking for $\mathcal{A}$} if there is an integer $n$ and $(a_{ij}) \in M_n(\mathcal{A})$ so that for any other irreducible representation $\omega$ of $\mathcal{C}$ not equivalent to $\pi$, $\|(\pi(a_{ij}))\|_n > \|(\omega(a_{ij}))\|_n$. In this case, we say that $\pi$ \emph{peaks at $(a_{ij})$}.
	\end{defn}
	While peaking representations were originally defined for operator systems rather than operator algebras, peaking in this sense implies peaking in the sense of \cite{Arv2008a} by \cite[Prop.~1.2.8]{Arv1969a}.
	We care about peaking representations because they are concrete examples of boundary representations:

\begin{theorem} \cite[Thm.~3.1, Rmk.~3.5]{Kleski2014} Let $\mathcal{A}$ be a separable operator algebra. For any $(a_{ij}) \in M_n(\mathcal{A})$, there is a boundary representation $\pi$ for $\mathcal{A}$ so that $\| (\pi (a_{ij})) \|= \|(a_{ij}) \|$. Thus every peaking representation is a boundary representation. \label{KleskiThm}
	\end{theorem}

We are now able to prove \cref{BoundaryRepsAreEval} from the introduction:

\begin{proof}[Proof of \cref{BoundaryRepsAreEval}]
We first show $\eval_{t \in \partial R}$ is a boundary representation of $\Gamma_c(\overline{R},\mathfrak{E}(\overline{R}))$ for $\Gamma_h(\overline{R},\mathfrak{E}(\overline{R}))$.

The points $t \in \partial R$ are peak points for $\mathcal{Z}(\Gamma_h(\overline{R},\mathfrak{E}(\overline{R})))\simeq A(\overline{R})$. Therefore there exists $(\sigma_U)_\mathcal{U} \in \mathcal{Z}(\Gamma_h(\overline{R},\mathfrak{E}(\overline{R})))$ such that $\eval_t(\sigma)$ and $||\eval_v(\sigma)||<1$ for all $v \not =t$.
	View $(\sigma_U)_{\mathcal{U}}$ as a $1\times 1$ matrix over $\Gamma_h(\overline{R},\mathfrak{E}(\overline{R}))$. Then $\eval_t$ is a peaking representation because any other irreducible representation of $\Gamma_c(\overline{R},\mathfrak{E}(\overline{R}))$ is also a point evaluation at some $s$.
	Thus $\eval_{t \in \partial R}$ is a boundary representation of $\Gamma_c(\overline{R},\mathfrak{E}(\overline{R}))$ for $\Gamma_h(\overline{R},\mathfrak{E}(\overline{R}))$ by \cref{KleskiThm}.

For the converse direction, first observe that, by \cref{RestrictionCompletelyIsom}, $\Gamma_h(\overline{R},\mathfrak{E}(\overline{R})) \simeq \Gamma_h(\partial R,\mathfrak{E}(\partial R))$. Then \cref{BoundaryRepCorrespondence} implies that for each boundary representation $\nu$ of $\Gamma_c(\overline{R},\mathfrak{E}(\overline{R}))$ for $\Gamma_h(\overline{R},\mathfrak{E}(\overline{R}))$, there is a boundary representation $\nu_1$ of $\Gamma_c(\partial R,\mathfrak{E}(\partial R))$ for $\Gamma_h(\partial R,\mathfrak{E}(\partial R))$, such that $\nu_1 \circ res (\sigma_U)= \nu(\sigma_U)$ for all $(\sigma_U) \in \Gamma_h(\overline{R},\mathfrak{E}(\overline{R}))$. But the $C^*$-representations of $\Gamma_c(\overline{R},\mathfrak{E}(\overline{R}))$ and $\Gamma_c(\partial R,\mathfrak{E}(\partial R))$ are all point evaluations, so we must conclude that $\nu$ is evaluation in $\partial R$, as well.
	\end{proof}

\begin{corollary} The noncommutative Shilov boundary of $\Gamma_h(\overline{R},\mathfrak{E}(\overline{R}))$ for $\Gamma_c(\overline{R},\mathfrak{E}(\overline{R}))$ is $\{ \sigma \in \Gamma_c(\overline{R},\mathfrak{E}) \mid \sigma(\partial R)=0 \}$, and the $C^*$-envelope is $\Gamma_c(\partial R,\mathfrak{E}(\partial R))$.
	\end{corollary}
\begin{proof}
We know that $\mathcal{I}(\Gamma_h(\overline{R},\mathfrak{E}(\overline{R})))$ is the intersection of the kernels of boundary representations, which will be the intersection of the kernels of $\eval_t$, $t \in \partial R$. This is precisely the collection of sections $\sigma \in \Gamma_c(\overline{R},\mathfrak{E}(\overline{R}))$ with $\sigma(\partial R)=0$; we denote this set $\Gamma_{c,0}(\overline{R},\mathfrak{E})$. Thus $\displaystyle C^*_e(\Gamma_h(\overline{R},\mathfrak{E}(\overline{R})))= \Gamma_c(\overline{R},\mathfrak{E}(\overline{R})) / \Gamma_{c,0}(\overline{R},\mathfrak{E}(\overline{R}))$, and $C^*_e(\Gamma_h(\overline{R},\mathfrak{E}(\overline{R})))$ is $*$-isomorphic to $ \Gamma_c(\partial R,\mathfrak{E}(\partial R))$ via the map $\varphi : \Gamma_c(\partial R,\mathfrak{E}(\partial R)) \to\Gamma_c(\overline{R},\mathfrak{E}(\overline{R})) / \Gamma_{c,0}(\overline{R})$, where $\varphi((\sigma_U|_{\partial R})_{\mathcal{U}}) = (\sigma_U)_\mathcal{U} + \Gamma_{c,0}(\overline{R})$.
	\end{proof}

	\end{section}

\begin{section}{Azumaya algebras} \label{AzumayaAlgebrasSection}

A unital algebra $\mathcal{A}$ is called an \emph{Azumaya} $\mathcal{Z}$-algebra if the center of $\mathcal{A}$ is $\mathcal{Z}$, $\mathcal{A}$ is a finitely generated projective $\mathcal{Z}$-module, and $\mathcal{A} \odot_{\mathcal{Z}} \mathcal{A}^{opp} \simeq \End_\mathcal{Z} \mathcal{A}$. Here $\odot_\mathcal{Z}$ is the algebraic $\mathcal{Z}$-module tensor product balanced over $\mathcal{Z}$, and the isomorphism is given by the formula $a \odot_Z b \to (c \mapsto acb)$. Azumaya algebras are also called central separable algebras, and as that name suggests they are generalizations of central simple algebras. There are several equivalent formulations of Azumaya algebras, and we will be interested in the one described in \cite{Procesi1972} (\cref{ArtinProcesiTheorem} below), with which it is much easier to do computations.

A \emph{polynomial identity of degree $d$ for $M_n(\mathbb{C})$} is an element $p$ of the (unital) free algebra on $d$ generators so that for any $d$-tuple $(A_1, \ldots, A_d) \in M_n(\mathbb{C})^d$, we have $p(A_1, \ldots, A_d)=0$ in $M_n(\mathbb{C})$. The algebra of \emph{$d$-generic $n \times n$ matrices} is the quotient of the free algebra on $d$ generators by the ideal of polynomial identities for the $n \times n$ matrices. A \emph{central polynomial} is an element of the $d$-generic $n \times n$ matrices which is not a polynomial identity for the $n \times n$ matrices, and yields an element in the center of $M_n(\mathbb{C})$ when evaluated at any $d$-tuple in $M_n(\mathbb{C})^d$.

\begin{defn} Let $\mathcal{A}$ be a unital ring that satisfies all polynomial identities of the $n \times n$ matrices. The Formanek center of $\mathcal{A}$, $\mathcal{F}(\mathcal{A})$, is the subset of the center of $\mathcal{A}$ obtained by evaluating the central polynomials with no constant term (considered as an elements of the free algebra on $d$ generators), at elements in $\mathcal{A}^d$.
	\end{defn}
The following theorem is a combination of two theorems, first \cite[Thm.~8.3]{Artin1969} of Artin, and second \cite[Thm.~3.2]{Procesi1972} of Procesi.

\begin{theorem}[Artin-Procesi Theorem] Let $\mathcal{A}$ be a unital ring that satisfies all the polynomial identities of the $n \times n$ matrices. Then $\mathcal{A}$ is an Azumaya algebra iff $\mathcal{A}=F(\mathcal{A}) \mathcal{A}$. \label{ArtinProcesiTheorem}
	\end{theorem}

We are now able to prove \cref{HoloSectionsAlgIsAzumaya}:

\begin{proof}[Proof of \cref{HoloSectionsAlgIsAzumaya}]
Let $\mathcal{U}$ be an open cover of $\breve{R}$, which, along with a cocycle in $H^1(\mathcal{U},\mathcal{PU}_n(\mathbb{C})_h)$, determines the bundle $\mathfrak{E}(\overline{R})$. It's clear that any section $(\sigma_U)_\mathcal{U}$ of the matrix $PU_n(\mathbb{C})$-bundle $\mathfrak{E}(\overline{R})$ will satisfy all polynomial identities of $M_n(\mathbb{C})$, so we check the second condition of \cref{ArtinProcesiTheorem}.

If we can find an integer $N$, elements $f_i \in F(\mathcal{A})$, $i=1,...,N$, and elements $g_i \in \mathcal{Z}$, $i=1,...,N$, so that $ \sum_{i=1}^N f_i g_i =1$, then we will have proven $\mathcal{A}$ is Azumaya. By \cite[Thm.~6.1]{Arens1958} and the compactness of $\overline{R}$, it's enough to find a collection of elements $f_i \in \mathcal{F}(\mathcal{A})$ with no common zero. We remind the reader that an element $f \in \mathcal{F}(\mathcal{A})$ is given by the following formula: for an open set $U\subseteq \overline{R}$, $r \in U$, and some central polynomial $c$, $f(r)=:\eval_{\mathcal{U},U,r}(f)=c(\sigma_{U,1}(r),\ldots, \sigma_{U,{K}}(r))$ for some tuple of continuous holomorphic sections $(\sigma_{U,i})_\mathcal{U} \in \mathcal{A}$, $i=1,\ldots K$. The value of $f(r)$ is independent of the open set $U$ since a central polynomial takes values in the center of $M_n(\mathbb{C})$, $\mathbb{C}I_n$, which is unaffected by the transition functions.

Formanek \cite{Formanek1972} and Razmyslov \cite{Razmyslov1973} independently discovered a special central polynomial $c_n$ for $M_n(\mathbb{C})$. Each polynomial $c_n$ is of $2n+1$ matrix variables $c_n=c_n(X_1,\cdots,X_{2n+1})$, $c_n$ has no constant term, and $c_n$ is not identically 0. More specifically, the central polynomial $c_n$ evaluates to 1 at a $(2n+1)$-tuple of matrices where the entries are chosen to be a particular sequence of matrix units $\{E_{ij} \mid 1 \leq i,j \leq n \}$. Thus, if for each $r \in \overline{R}$ and neighborhood $U$ containing $r$ we can construct a holomorphic section $(\sigma_U)_\mathcal{U}$ of $\mathfrak{E}(\breve{R})$ so that $\eval_{\mathcal{U},U,r}(\sigma_U) = E_{ij}$, then we can build an element $f_r \in \mathcal{F}(\mathcal{A})$ with $f_r(r)=1$. But the existence of such sections is guaranteed by \cref{ExistenceOfSections}. To conclude our proof, we use the elements $f_r$ to build a collection of elements of $\mathcal{F}(\mathcal{A})$ with no common zero. This is done as follows.

The elements $f_r$ are holomorphic on $\breve{R}$, and so have finitely many zeros in $\overline{R}$. Pick $f_{r_0}$ with zeros $x_{01}, x_{02}, \ldots, x_{0n_0}$. We then sequentially pick elements $f_{r_i}=f_{x_{0i}}$, which by definition satisfy $f_{x_{0i}}(x_{0i})=1 \not = 0$. Thus this collection of elements in $\mathcal{F}(\mathcal{A})$ has no common zero, and we are done.
	\end{proof}

\begin{corollary} The two-sided ideals of $\Gamma_h(\overline{R},\mathfrak{E}(\overline{R}))$ are in 1-1 correspondence with the ideals of $A(\overline{R})$.
	\end{corollary}
\begin{proof} See, for example, \cite[Cor.~3.7]{DeMeyer1971}.
	\end{proof}

There is a Banach-algebraic notion of Azumaya algebra \cite{Craw1983}, which coincides with the purely algebraic definition when the algebras in question are unital. Thus we also have $\Gamma_h(\overline{R},\mathfrak{E}(\overline{R}))$ is a Banach Azumaya algebra in the sense of \cite[Prop.~2.6, Thm.~2.11]{Craw1983}.

	\end{section}

\begin{section}{Appendix}
For a $C^\infty$ real manifold \cite[Sec.~2]{Kobayashi1987}, or a Riemann surface \cite[Sec.~6]{Gunning1966}, one can define a flat vector bundle as a vector bundle which has a coordinate representative with locally constant transition functions. This is equivalent to the following perspective. Let \underline{$GL_n(\mathbb{C})$} be the constant sheaf over $S$, $\mathcal{GL}_n(\mathbb{C})_h$ the sheaf of germs of holomorphic $GL_n(\mathbb{C})$-valued functions, $i : \mbox{\underline{$GL_n(\mathbb{C})$}} \to \mathcal{GL}_n(\mathbb{C})_h$ the inclusion of sheaves, and $i_*: H^1(S,\mbox{\underline{$GL_n(\mathbb{C})$}}) \to H^1(S,\mathcal{GL}_n(\mathbb{C})_h)$ the induced map. Then a vector bundle with locally constant transition functions can be regarded as an element in the image of $i_*$ (See, for example, \cite[Sec.~6]{Gunning1966}).
	\end{section}

\begin{proof}[Proof of \cref{FlatPrincipalBundleAsRep}] Let's first consider the holomorphic case. In \cite[Lem.~27]{Gunning1967}, Gunning gives a 1-1 correspondence between the sets $\Hom(\pi_1(S),U_n)/U_n$ and $H^1(S,\mbox{\underline{$U_n(\mathbb{C})$}})$. The same argument follows for the group $PU_n(\mathbb{C})$. Gunning's proof quite clearly and visually connects an element of $\Hom(\pi_1(S),G)/G$ with a coordinate bundle associated to a particular open cover $\mathcal{U}$ of $S$. The continuous case is addressed in \cite[Thm.~13.9]{Steenrod1951}. For a discrete group $G$, Steenrod outlines the 1-1 correspondence between elements of $\Hom(\pi_1(S),G)/G$ and elements in $H^1(S,\mbox{\underline{$G$}})$, for any topological space $S$ that is locally compact, Hausdorff, second-countable, path-connected, locally path-connected, and semilocally simply connected.
	\end{proof}

\begin{proof}[Proof of \cref{FlatMatrixBundleAsRep}] The map $i_* : H^1(S,\mbox{\underline{$PU_n (\mathbb{C})$}}) \to H^1(S,\mathcal{PU}_n(\mathbb{C})_h)$ will be injective (in fact, an isomorphism), since an equivalence between coordinate bundles, viewed as elements in $Z^1(S,\mathcal{PU}_n(\mathbb{C})_h)$, must be implemented by a cochain $(\lambda_U)_\mathcal{U}$ in which each $\lambda_U$ is holomorphic and $PU_n(\mathbb{C})$-valued -- that is, locally constant.
	\end{proof}	
	
\begin{proof}[Proof of \cref{ExtensionOfBundleToRegularRegion}]
We follow the proof in \cite[pg.~306]{Widom1971}, substituting \cref{FlatMatrixBundleAsRep} for the analogous correspondence for vector bundles used by Widom.
	\end{proof}

\begin{section}{Acknowledgement}
The author would like to thank Paul Muhly for his incisive and insightful feedback, as well as his lively encouragement, during the preparation of this work.
	\end{section}
	
\input{KMcCormick_Matrix_Bundles.bbl}

\bibliographystyle{plain}

	\end{document}

%% file: KMcCormick_Matrix_Bundles.bbl
\def\cprime{$'$}